\documentclass[preprint,12pt]{elsarticle}
\usepackage{}
\usepackage{enumitem}
\usepackage{amsmath,amssymb,enumitem}
\usepackage{amsfonts}

\usepackage{amsfonts}
\usepackage{amsmath}
\usepackage{amssymb}
\usepackage{graphics}
\usepackage{graphicx}
\usepackage{mathrsfs}
\usepackage{indentfirst}
\usepackage{stmaryrd}
\usepackage{latexsym}
\usepackage{xypic}
\newtheorem{definition}{Definition}[section]
\newtheorem{lemma}[definition]{Lemma}
\newtheorem{theorem}[definition]{Theorem}
\newtheorem{proposition}[definition]{Proposition}
\newtheorem{example}[definition]{Example}

\newtheorem{remark}[definition]{Remark}
\newtheorem{corollary}[definition]{Corollary}
\newtheorem{open problem}[definition]{Open problem}
\newproof{proof}{\textbf{Proof}}

\begin{document}

\begin{frontmatter}



\title{\textbf{Filter-induced linear topologies on residuated lattices: Hausdorffness, profiniteness, and finiteness conditions}}


\author{Jiang Yang$^{a}$, Pengfei He$^{b,\ast}$, Juntao Wang$^{c}$ }
\cortext[cor1]{Corresponding author. \\
Email addresses: yangjiangdy@126.com, hepengf1986@126.com, wjt@xsyu.edu.cn }

\address[A]{School of Mathematical Sciences, Guangxi Minzu University, Nanning, 530006, China}
\address[B]{School of Mathematics and Information Science, Shaanxi Normal University, Xi'an, 710119, China}
\address[C]{School of Science, Xi'an Shiyou University, Xi'an, 710065, China}

\begin{abstract}

We study linear topologies on residuated lattices generated by systems of filters, with emphasis on the uniform structures and separation properties that they determine. A down-directed family of filters gives a natural compatible uniformity, and the associated topology makes the residuated lattice into a topological algebra. We characterize Hausdorffness by the triviality of the intersection of the underlying filter system. For compact topological residuated lattices, we prove the equivalence between topological profiniteness, residual finiteness, and representation as a closed subdirect product of finite discrete residuated lattices. We also analyze the descending chain condition ($DCC$) on filters. Under $DCC$, every filter system has a least element; hence every zero-dimensional linear topology is induced by a single filter, and the canonical map from filters to zero-dimensional linear topologies is bijective. This gives a corrected form of earlier representation arguments and identifies precisely where $DCC$ is required. Finally, working throughout in $\mathrm{ZFC}$, we give a sufficient criterion for the existence of non-discrete Hausdorff linear topologies, illustrated by the G\"{o}del algebra.

\end{abstract}


\end{frontmatter}


\section{Introduction}
\label{intro}

Topological algebra studies algebraic structures equipped with topologies for which the basic operations are continuous. A classical and particularly useful source of such topologies is provided by linear topologies. In the case of groups, a suitable family of subgroups may be used as a neighbourhood base at the identity; Birkhoff \cite[pp. 52--54]{Birkhoof} already emphasized this point of view. Such ``subgroup topologies''---or, more generally, linear topologies---are naturally described by uniform structures; see, for example, \cite{Dimitric,Fuchs,Hoo,Hull}. From this perspective, the main question is not merely whether a compatible topology exists, but how its separation, compactness, and finite-quotient properties are reflected by the algebraic data defining the uniformity.

This paper applies that uniform-theoretic viewpoint to residuated lattices. Residuated lattices were introduced in the algebraic literature by Ward and Dilworth \cite{Ward1} and include many ordered algebraic structures such as lattices of ideals, lattice-ordered groups, and algebras of relations. They also play an important role in algebraic logic, but the emphasis of the present paper is topological: we regard a residuated lattice as an algebra whose filters provide a canonical family of congruences and hence a natural source of uniformities. For a residuated lattice $L$, filters correspond to congruences, and therefore a down-directed family of filters is expected to behave like a system of neighbourhoods of the distinguished element $1$.

More precisely, if $F$ is a filter of $L$, the associated congruence classes are denoted by $x/F$. The topology
\[
\mathcal{T}_{F}=\{U\subseteq L: \text{for every }x\in U,\ x/F\subseteq U\}
\]
is a basic example of a filter-induced linear topology. More generally, a system of filters $\mathcal F$ produces the topology
\[
\mathcal{T}_{\mathcal F}=\{U\subseteq L: \text{for every }x\in U,\text{ there is }F\in\mathcal F\text{ such that }x/F\subseteq U\}.
\]
The corresponding family of congruences is a base for a uniformity, and the resulting topological space is compatible with all algebraic operations. This construction is the organizing theme of the paper.

The first issue is Hausdorffness. For a uniform space, the Hausdorff property is equivalent to the diagonal being the intersection of all entourages. In the present setting this condition becomes a purely filter-theoretic statement: the topology induced by a system of filters is Hausdorff exactly when the intersection of that system is the trivial filter $\{1\}$. This gives a direct bridge between separation axioms and subdirect representations of residuated lattices.

The second issue is profiniteness. Profinite topological algebras are inverse limits of finite discrete algebras. For topological groups, rings, semigroups, and distributive lattices, profiniteness is closely related to Stone-type topological properties; see Johnstone \cite[Sec. VI.2]{Johnstone} and Ribes--Zalesskii \cite{Ribes}. Banaschewski's work on profinite universal algebras \cite{Banaschewski} provides a general background for this viewpoint. For residuated lattices, however, profiniteness has not been studied as systematically. We prove that, in the compact case, topological profiniteness, residual finiteness, and representation as a closed subdirect product of finite discrete residuated lattices are equivalent. Thus compactness and finite quotients interact in the expected topological-algebraic way.

The third issue is the representation of zero-dimensional linear topologies. Earlier work on topological residuated lattices asserted broad representation statements for such topologies. The subtle point is that passing from a linear topology to a filter system involves finite intersections of filter neighbourhoods, while arbitrary intersections need not remain inside the same system. This is precisely where an additional finiteness condition is needed. We show that the descending chain condition ($DCC$) on filters supplies the missing step: every system of filters then has a least element. Consequently, on a residuated lattice satisfying $DCC$, every zero-dimensional linear topology is induced by a single filter. In particular, the canonical map from filters to zero-dimensional linear topologies is bijective under $DCC$.

The fourth issue is the existence of non-discrete Hausdorff compatible topologies. Throughout the paper we work in $\mathrm{ZFC}$. The analogous group-theoretic problem is the topologizability problem: an infinite group need not admit a non-discrete Hausdorff group topology. Infinite non-topologizable groups exist in $\mathrm{ZFC}$; see Hesse \cite{Hesse}, Olshanskii \cite{Olshanskii}, and Klyachko--Olshanskii--Osin \cite{KlyachkoOlshanskiiOsin}. Hence the existence of a non-discrete Hausdorff topology is not a formal consequence of infinitude. In the residuated-lattice setting we give a filter-system criterion for the existence of non-discrete Hausdorff linear topologies and illustrate it by the G\"{o}del algebra.

The paper is organized as follows. Section 2 recalls the necessary facts on residuated lattices, filters, congruences, and uniform structures, and proves the general uniform construction from down-directed congruence families. Section 3 studies profinite residuated lattices and proves the compact profiniteness criteria. Section 4 develops the finiteness conditions on filters, especially $DCC$, and records their consequences for filter systems. Section 5 applies these results to Hausdorff and zero-dimensional linear topologies, corrects the relevant representation argument under $DCC$, and gives examples and sufficient criteria for non-discrete Hausdorff linear topologies.

\section{Preliminaries}
In this section, we summarize some basic definitions and results, which will be used.


\begin{definition} \label{2.1a}
\emph{\cite{Hohle} An algebraic structure $\mathbf{L}=(L,\wedge,\vee,\odot,\rightarrow,0,1)$ of type $(2,2,2,\\2,0,0)$ is called a \emph{residuated lattice} if it satisfies the following conditions:
\begin{itemize}
 \item[\rm (1)] $(L,\wedge,\vee,0,1)$ is a bounded lattice;
 \item[\rm (2)] $(L,\odot,1)$ is a commutative monoid;
 \item[\rm (3)] $x \odot y \leq z$ if and only if $x \leq y \rightarrow z$, for all $x,y,z \in L$, where $\leq$ is the partial order of the lattice $(L,\wedge,\vee, 0, 1)$.
 \end{itemize}}
\end{definition}

\begin{remark}\label{2.1b}
For many readers, a \emph{residuated lattice} is a more general structure which lacks commutativity of monoid reduct (and hence requires two ``implication'' or ``division'' operations in the signature: $\rightarrow,\leftarrow$, or also denoted $\backslash,/$) and bounds for the lattice reduct, and moreover does not include a designated constant ``0''. That is, a \emph{residuated lattice} is an algebra
$\mathbf{L}=(L,\wedge,\vee,\odot,\backslash,/,1)$ where $(L,\wedge,\vee)$ is a lattice, $(L,\odot,1)$ is a monoid, and satisfies the \emph{law of residuation}: for all $x,y,z\in L$
$$x\leq z/y\Leftrightarrow x\odot y\leq z\Leftrightarrow y\leq x\backslash z,$$
where $\leq$ is the induced lattice order defined by $x\leq y$ iff $x\vee y=y$. An \emph{$FL$-algebra, or pointed residuated lattice}, is simply a residuated lattice with an additional designated constant 0 in the signature
; i.e., residuated lattices are the 0-free reduct of $FL$-algebras. $\mathbf{L}$ is called \emph{commutative} if the monoid reduct $(L,\odot,1)$ is commutative (in which case, $x\backslash y=y/x$ and whose common value is
often denoted $x\rightarrow y$); \emph{integral} if the monoid identity 1 is the greatest element in the lattice reduct $(L,\wedge,\vee)$; \emph{bounded} if there is a least element $\bot$ in the lattice reduct (which
implies there is also greatest element $\top:=\bot\backslash\bot$ by residuation), where for an $FL$-algebra, when $\bot=0$ it is called \emph{zero-bounded}. In this paper, what we call \emph{residuated lattice} is often understood as an $FL_{ew}$-algebra, but we use the former term for brevity.
\end{remark}


For convenience of readers, we provide some basic properties of residuated lattices in the following proposition.
\begin{proposition} \label{2.5a}\cite{Turunen}
In any residuated lattice $(L,\wedge,\vee,\odot,\rightarrow,0,1)$, the following properties hold: for any $x, y, z \in L$,
\begin{enumerate}
  \item[\rm$(R_{1})$] $1\rightarrow x=x$, $x \rightarrow 1 = 1$;
  \item[\rm$(R_{2})$] $x \leq y$ if and only if $x \rightarrow y =1$;
  \item[\rm$(R_{3})$] if $x \leq y$, then $y \rightarrow z \leq x \rightarrow z$, $z \rightarrow x \leq z \rightarrow y$ and
              $x \odot z \leq y \odot z$;
  \item[\rm$(R_{4})$] $x \odot (x \rightarrow y) \leq y$;
  \item[\rm$(R_{5})$] $x \odot y \leq x \wedge y$, $x \leq y \rightarrow x$;
  \item[\rm$(R_{6})$] $x \rightarrow (y \rightarrow z)=(x \odot y) \rightarrow z = y \rightarrow (x \rightarrow z)$;
  \item[\rm$(R_{7})$] $x\vee(y\odot z)\geq (x\vee y)\odot(x\vee z)$, hence $x^{m}\vee y^{n}\geq(x\vee y)^{mn}$.
\end{enumerate}
\end{proposition}

\begin{definition} \label{2.6a}\emph{\cite{Turunen} Let $(L,\wedge,\vee,\odot,\rightarrow,0,1)$ be a residuated lattice. A \emph{filter} is a nonempty set $F\subseteq L$ such that for each $x, y\in L$,
\begin{enumerate}
\item[\rm(i)] $x,y\in F$ implies $x\odot y\in F$,
\item[\rm(ii)] if $x\in F$ and $x\leq y$, then $y\in F$.
\end{enumerate}}
\end{definition}

Note that in a residuated lattice $L$, a filter $F$ of $L$ is equivalent to a \emph{deductive system}, that is, $F$ satisfies the following conditions: (i) $1\in F$, and (ii) $x,~x\rightarrow y\in F$ implies $y\in F$.

Let $L$ be a residuated lattice. By $\mathcal{F}(L)$ and $Con(L)$, we mean the set of all filters (congruences) of $L$. There is a close correspondence between congruences and filters of residuated lattices. For each congruence
$\theta$ on the residuated lattice $L$, let $[1]_{\theta}=\{x\in L:(1,x)\in \theta\}$. Then $[1]_{\theta}$ is a filter of $L$, called the filter \emph{determined} by a congruence $\theta$. Conversely, for each filter $F$,
$\theta_{F}=\{(x,y)\in L\times L:(x\rightarrow y)\odot(y\rightarrow x)\in F\}$ is a congruence on $L$, called the congruence \emph{determined} by a filter $F$. For any $x\in L$, let $x/F$ be the equivalence class $x/\theta_{F}$. If we denote by $L/F$ the quotient set $L/\theta_{F}$, then $L/F$ becomes a residuated lattice with the operations induced from those of $L$. Moreover, the following result holds.

\begin{theorem}\label{4.1}\cite{Galatos} For every residuated lattice $L$, we have $$\mathcal{F}(L)\cong Con(L).$$
In other words, the lattice of filters is isomorphic to the lattice of congruences over $L$.
\end{theorem}

We now recall the notions of \emph{uniform structures} and \emph{uniform topologies}. We shall use the following notation. \\

Let $X$ be a set. We denote by $\Delta_{X}$ the diagonal
in $X\times X$, namely $\Delta_{X}=\{(x,x):x\in X\}\subset X\times X$. Suppose that $R$ is a subset of $X\times X$ (in other words, $R$ is a binary relation on $X$). For any $y\in X$, we define the set
$R[y]=\{x\in X:(x,y)\in R\}$. The \emph{inverse} $R^{-1}\subset X\times X$ of $R$ is defined by $R^{-1}=\{(x,y):(y,x)\in R\}$. One says that $R$ is symmetric if it satisfies $R^{-1}=R$. If $R$ and $S$ are subsets of
$X\times X$, we define their \emph{composite} $R\circ S\subset X\times X$ by $R\circ S=\{(x,y): \exists z\in X~ s.t.~(x,z)\in S~ and~ (z,y)\in R\}$.

\begin{definition}\label{2.5}\emph{\cite{Jam} Let $X$ be a set. A \emph{uniform structure} on $X$ is a non-empty set $\mathcal{U}$ of subsets of $X\times X$ satisfying the following conditions:
\begin{itemize}
 \item[\rm (1)] if $V\in \mathcal{U}$, then $\Delta_{X}\subseteq V$;
 \item[\rm (2)] if $V\in\mathcal{U}$ and $V\subseteq V'\subseteq X\times X$, then $V'\in \mathcal{U}$;
 \item[\rm (3)] if $V, W\in\mathcal{U}$, then $V\cap W\in\mathcal{U}$;
 \item[\rm (4)] if $V\in\mathcal{U}$, then $V^{-1}\in\mathcal{U}$;
 \item[\rm (5)] if $V\in \mathcal{U}$, then there exists $W\in \mathcal{U}$ such that $W\circ W\subseteq V$.
\end{itemize}}

\emph{A set $X$ equipped with a uniform structure $\mathcal{U}$ is called a \emph{uniform space} and the elements of $\mathcal{U}$ are called the \emph{entourages} of $X$. The topology associated with $\mathcal U$ is defined as follows: a subset $\Omega\subseteq X$ is open if, for each $x\in\Omega$, there exists an entourage $V\in\mathcal U$ such that $V[x]\subseteq \Omega$. Equivalently, a subset $N\subseteq X$ is a neighbourhood of $x$ if and only if there exists $V\in\mathcal U$ such that $V[x]\subseteq N$. This topology is Hausdorff if and only if the intersection of all entourages is the diagonal $\Delta_X$.}
\end{definition}

\begin{proposition}\label{2.6}\cite{Jam} Let $X$ be a set and let $\mathcal{B}$ be a nonempty set of subsets of $X\times X$. Then $\mathcal{B}$ is a base for some (necessarily unique) uniform structure on $X$ if and only
if it satisfies the following properties:
\begin{itemize}
 \item[\rm (a)] if $V\in \mathcal{B}$, then $\Delta_{X}\subset V$;
 \item[\rm (b)] if $V\in\mathcal{B}$ and $W\in\mathcal{B}$, then there exists $U\in\mathcal{B}$ such that $U\subset V\cap W$;
 \item[\rm (c)] if $V\in\mathcal{B}$, then there exists $W\in\mathcal{B}$ such that $W\subset V^{-1}$;
 \item[\rm (d)] if $V\in\mathcal{B}$, then there exists $W\in\mathcal{B}$ such that $W\circ W\subset V$.
\end{itemize}
\end{proposition}

A \emph{topological (universal) algebra} is an object $\mathbf{A}=(A,(f_{\alpha})_{\alpha\in I},\mathcal{T})$ where $A$ is a set, $(f_{\alpha})_{\alpha\in I}$ a family of maps $f_{\alpha}:A^{n_{\alpha}}\rightarrow A$, $n_{\alpha}$
the
\emph{arity} of the operation $f_{\alpha}$, and $\mathcal{T}$ a topology on $A$ such that each $f_{\alpha}$ maps the product space $(A,\mathcal{T})^{n_{\alpha}}$ continuously to the space $(A,\mathcal{T})$. $(A,(f_{\alpha})_{\alpha\in I})$ is
called the \emph{underlying algebra}, and $(A,\mathcal{T})$ the \emph{underlying space}. In what follows, we give a general result about topological algebras.\\

Recall that a poset $(P,\leq)$ is called a \emph{down-directed ordered set} if for every two different elements $a,b\in P$ there exists a lower bound of the two-element set $\{a,b\}$.
By Proposition \ref{2.6}, we have the following result, which is of great significance to study the topological algebras.

\begin{theorem}\label{2.7} Let $\mathbf{A}=(A,(f_{\alpha})_{\alpha\in I})$ be an algebra and $\mathscr{C}\subseteq Con(\mathbf{A})$. If $(\mathscr{C},\subseteq)$ is a down-directed set, then $\mathscr{C}$ is a base for a uniformity on $A$. The associated topology $\mathcal{T}$ makes $(A,(f_{\alpha})_{\alpha\in I},\mathcal{T})$ a topological algebra.
\end{theorem}
\begin{proof}
Each congruence in $\mathscr C$ contains the diagonal and is symmetric. Since $(\mathscr C,\subseteq)$ is down-directed, for $\theta,\psi\in\mathscr C$ there exists $\eta\in\mathscr C$ such that $\eta\subseteq\theta\cap\psi$. Finally, $\theta\circ\theta=\theta$ for every equivalence relation $\theta$, so condition (d) of Proposition \ref{2.6} is also satisfied. Thus $\mathscr C$ is a base for a uniformity on $A$.

Let $f=f_\alpha$ have arity $n$. Fix $a=(a_1,\ldots,a_n)\in A^n$ and let $\theta\in\mathscr C$. Since $\theta$ is a congruence,
\[
 f(\theta[a_1]\times\cdots\times\theta[a_n])\subseteq \theta[f(a_1,\ldots,a_n)].
\]
The sets $\theta[b]$ $(\theta\in\mathscr C)$ form a neighbourhood base at each $b\in A$. Hence $f$ is continuous at $a$. Since $a$ and $f$ were arbitrary, all fundamental operations are continuous.
\end{proof}

The variety $\mathcal{RL}$ of residuated lattices is arithmetical. It has the congruence extension property (CEP), and is congruence 1-regular, i.e., for any congruence $\theta$, the coset of $1$ uniquely determines $\theta$ \cite[Chapter 2, pp. 86--89]{Galatos}. Furthermore, the congruences of a residuated lattice are completely determined by filters (see Theorem \ref{4.1} or \cite[Theorem 2.3.8]{Galatos}). Thus, in \cite{Yang1}, we studied topological residuated lattices induced by a system of filters. By Theorem \ref{2.7}, it follows that a residuated lattice endowed with the topology induced by a system of filters is a topological residuated lattice, see \cite[Theorem 3.13]{Yang1}.\\

A topological space $(X,\mathcal{T})$ is a \emph{zero-dimensional space} if $\mathcal{T}$ has a base which consists of clopen (i.e., open-and-closed) subsets.
Let $\mathcal{T}$ and $\mathcal{T}'$ be two topologies on a given set $X$. If $\mathcal{T}\subseteq\mathcal{T}'$, then we say that $\mathcal{T}'$ \emph{finer than} $\mathcal{T}$.

\begin{definition}\label{2.1}\emph{\cite{Yang1} Let $\mathcal{F}$ be a family of filters of a residuated lattice $L$. Then $\mathcal{F}$ is called a \emph{system of filters} of $L$ if $(\mathcal{F},\subseteq)$ is
a down-directed set.}
\end{definition}

 A topology $\tau$ of a residuated lattice $L$ is called a \emph{linear topology} on $L$ if there exists a base $\beta$ for $\tau$ such that, for all $B\in\beta$, $B$ is a filter if and only if $1\in B$.

If $(\mathcal{F},\subseteq)$ is a system of filters of a residuated lattice $L$, then, by Proposition \ref{2.6} or \cite[Theorem 3.13]{Yang1}, $$\mathcal{T}_{\mathcal{F}}=\{U\subseteq L:\forall x\in U,\exists F\in\mathcal{F}~such~that~x/F\subseteq U\}$$
is a linear topology on $L$. We call it the \emph{linear topology induced by $\mathcal{F}$}. Also, $(L,\mathcal{T}_{\mathcal{F}})$ is a topological residuated lattice.


\section{Profinite residuated lattices}

In this section, we study profinite residuated lattices and give several characterizations of profiniteness in compact topological residuated lattices.\\

Let $I=(I,\leq)$ denote a \emph{directed partially ordered set} or \emph{directed poset}, that is, $I$ is a set with a binary relation $\leq$ such that $(I,\leq)$ is a poset and if $i,j\in I$, there exists some $k\in I$ such that
$i,j\leq k$.
\begin{definition}\label{3.1}\emph{\cite{Gratzer1}
(i) By an \emph{inverse (or projective) system} in a category $\mathcal{D}$ we mean a family $\{B_{i},\pi_{ij},I\}$ of objects, indexed by a directed poset $I$, with a family of morphisms $\pi_{ij}: B_{i}\rightarrow B_{j}$, for any $j\leq i$, satisfying the following conditions:
\begin{itemize}
\item[\rm (1)] $\pi_{ik}=\pi_{jk}\circ \pi_{ij}$, for any $k\leq j\leq i$;
\item[\rm (2)] $\pi_{ii}=id_{B_{i}}$, for any $i\in I$.
\end{itemize}
For brevity we say that $\{B_{i},\pi_{ij},I\}$ is an inverse system in $\mathcal{D}$.}

\emph{(ii) The \emph{inverse limit} of an inverse system $\{B_{i},\pi_{ij},I\}$ in a category $\mathcal{D}$ is an object $B$ of $\mathcal{D}$ together with a family $\{\phi_{i}:B\rightarrow B_{i}\}_{i\in I}$
of morphisms (which is often denoted by $\{B,\phi_{i}\}_{i\in I}$) satisfying the conditions:
\begin{itemize}
\item[\rm (1)] $\pi_{ij}\circ\phi_{i}=\phi_{j}$, for any $i,j\in I, j\leq i$(this condition often is called \emph{compatible condition});
\item[\rm (2)] for any object $B'$ of $\mathcal{D}$, together with a family of morphisms $\lambda_{i}: B'\rightarrow B_{i}, i\in I$, if $\pi_{ij}\circ\lambda_{i}=\lambda_{j}$, for any $i,j\in I, j\leq i$ then there exists a unique morphism $\lambda:B'\rightarrow B$ such that $\phi_{i}\circ\lambda=\lambda_{i}$, for any $i\in I$.
 \end{itemize}
    The inverse limit of the above system is denoted by $\lim\limits_{\longleftarrow}B_{i}$.}
\end{definition}

Recall from Gr\"atzer \cite{Gratzer1} that the inverse limits of families of algebras are constructed in the following way.
\begin{theorem}\cite{Gratzer1}\label{3.2}
 Let $\{B_{i},\pi_{ij},I\}$ be an inverse system of same algebras, $\prod_{i\in I}B_{i}$ be its product and $\pi_{j}:\prod_{i\in I}B_{i}\rightarrow B_{j}$ is defined by $\pi_{j}((x_{i})_{i\in I})=x_{j}$ for any $j\in I$. Let $$B=\{(b_{i})_{i\in I}\in\prod_{i\in I}B_{i}:\pi_{ij}(b_{i})=b_{j},j\leq i\}.$$
 Then $B$ is a subalgebra of $\prod_{i\in I}B_{i}$ and $\{B,\phi_{i}\}_{i\in I}$ is the inverse limit of $\{B_{i},\pi_{ij},I\}$, where $\phi_{i}=\pi_{i}\upharpoonright_{B}$ for any $i\in I$.
\end{theorem}

Let $A$ be an algebra. Let $I$ denote the set of all congruences $\theta$ on $A$ such that $A/\theta$ is finite. We denote the image of $a\in A$ in $A/\theta$ by $[a]_{\theta}$. If $\theta\subseteq\theta'$, then there is a canonical projection $\varphi_{\theta'\theta}:A/\theta'\rightarrow A/\theta$ given by $\varphi_{\theta'\theta}([a]_{\theta'})=[a]_{\theta}$. Then $(I,\supseteq)$ is a directed set, and
$(A/\theta,\varphi_{\theta'\theta},I)$ is an inverse system of algebras. The \emph{profinite completion $\widehat{A}$} of $A$ is the inverse limit of $(A/\theta,\varphi_{\theta'\theta},I)$. By \cite{Gratzer1}, we may identify $\widehat{A}$ with the subalgebra of $\prod_{\theta\in I}A/\theta$ consisting of all $([a]_{\theta})_{\theta\in I}$ for which $\varphi_{\theta'\theta}([a]_{\theta'})=[a]_{\theta}$ whenever
$\theta\subseteq\theta'$. We define the canonical homomorphism $e_{A}:A\rightarrow \widehat{A}$ by $e_{A}(a)=([a]_{\theta})_{\theta\in I}$. Let $\pi_{\theta'}:\prod_{\theta\in I}A/\theta\rightarrow A/\theta'$ denote the projection map. We also denote the restriction of $\pi_{\theta'}$ to $\widehat{A}$ by $\pi_{\theta'}$.
One would expect a completion of an algebra $A$ to be an extension of $A$, however, the canonical homomorphism $e_{A}:A\rightarrow \widehat{A}$ is not always an embedding.

Let $(I,\leq)$ be a directed poset. Assume that $I'$ is a subset of $I$ in such a way that $(I',\leq)$ becomes a directed poset. We say that $I'$ is \emph{cofinal} in $I$ if for every $i\in I$ there is some $i'\in I'$ such that $i\leq i'$. If $\{X_{i},\varphi_{ij}, I\}$ is an inverse system and $I'$ is cofinal in $I$, then $\{X_{i},\varphi_{ij}, I'\}$ becomes an inverse system in an obvious way, and we say that $\{X_{i}, \varphi_{ij},I'\}$
is a cofinal subsystem of $\{X_{i},\varphi_{ij},I\}$.
\begin{proposition}\label{3.3}
Let $\{A_{i},\varphi_{ij},I\}$ be an inverse system of similar algebras over a directed poset $I$ and assume that $I'$ is a cofinal subset of $I$. Then $$\varprojlim_{i\in I}X_{i}\cong\varprojlim_{i'\in I'}X_{i'}.$$
\end{proposition}

\begin{proof}
Assume that $\{A_{i},\varphi_{ij},I'\}$ is a cofinal subsystem of $\{A_{i},\varphi_{ij},I\}$ and denote by $(\varprojlim_{i'\in I'},\varphi_{i'}')$ and $(\varprojlim_{i\in I},\varphi_{i})$ their corresponding inverse limits, respectively. For any $j\in I$, let $j'\in I'$ be such that $j\leq j'$. Define $$\overline{\varphi}_{j}:\varprojlim_{i'\in I'}X_{i'}\rightarrow X_{j}$$
as the composition of canonical mappings $\varphi_{j'j},\varphi_{j'}'$. Observe that the maps $\overline{\varphi}_{j}$ are well-defined and compatible. Hence they induce a map
$$\overline{\varphi}:\varprojlim_{i'\in I'}X_{i'}\rightarrow \varprojlim_{i\in I}X_{i}$$
such that $\varphi_{j}\overline{\varphi}=\overline{\varphi}_{j}(j\in I)$. We claim that the mapping $\overline{\varphi}$ is a bijection. Note that if $(x_{i'})\in \varprojlim_{i'\in I'}X_{i'}$ and
$\overline{\varphi}(x_{i'})=(y_{i})$, then $y_{i'}=x_{i'}$ for $i'\in I'$. Since $I'$ is cofinal in $I$, thus $\overline{\varphi}$ is an injection. To see that $\overline{\varphi}$ is a surjection, let $(y_{i})\in \varprojlim_{i\in I}X_{i}$ and consider the element $(x_{i'})$, where $x_{i'}=y_{i'}$ for every $i'\in I'$. Then $(x_{i'})\in \varprojlim_{i'\in I'}X_{i'}$ and obviously, $\overline{\varphi}(x_{i'})=(y_{i})$.
This proves the claim. The rest of the proof is obvious.
\end{proof}

In what follows we shall be specifically interested in the topological space $X$ that arises as inverse limit
\[X=\lim_{\substack{\longleftarrow\\ i\in I}}X_{i}\]
of finite spaces $X_{i}$ endowed with the discrete topology. We call such a space \emph{profinite space} or a \emph{Boolean space}. A topological space is \emph{totally disconnected} if every point in the space is its own connected component.

\begin{theorem}\label{3.4}\cite{Ribes}
Let $X$ be a topological space. Then the following conditions are equivalent.
\begin{enumerate}
  \item[(a)] $X$ is a profinite space;
  \item[(b)] $X$ is compact Hausdorff and totally disconnected;
  \item[(c)] $X$ is compact Hausdorff and admits a base of clopen sets for its topology.
\end{enumerate}
\end{theorem}

\begin{definition}\label{3.5}
\emph{A topological residuated lattice $(L,\tau)$ is called \emph{profinite} if it is topologically isomorphic to the inverse limit of an inverse system of finite discrete residuated lattices. When no topology is specified, a residuated lattice $L$ is called algebraically profinite if it is isomorphic, as a residuated lattice, to such an inverse limit equipped with its profinite topology. The class of algebraically profinite residuated lattices is denoted by $\mathbf{Pro}\mathcal{RL}$.}
\end{definition}

Thus, we see that if we are given a profinite algebra $A\cong\varprojlim_{i\in I}A_{i}$, then $A$ is a topological algebra in its profinite topology, which can be defined in a natural way using the limiting cone
$(\pi_{i}:A\rightarrow A_{i})_{i\in I}$. We will now record the fundamental fact that every profinite topology is Boolean, i.e., every profinite topology is compact Hausdorff and zero-dimensional.
\begin{proposition}\label{3.6}\cite{Banaschewski}
If $A\cong\varprojlim_{i\in I}A_{i}$ is a profinite algebra, then $A$ is a Boolean topological algebra in its profinite topology.
\end{proposition}

\begin{remark}\label{3.7}
We have chosen to first describe profinite algebras using universal algebra, i.e., as the limit of a diagram of finite algebras $\{A_{i},\varphi_{ij},I\}$, and then to indicate that one can define a topology on
$\varprojlim_{i\in I}A_{i}$, making $\varprojlim_{i\in I}A_{i}$ a topological algebra. Alternatively, one can start by endowing each algebra $A_{i}$ with the discrete topology, and then show that one can also take limits of topological algebras, so that it follows immediately that $\varprojlim_{i\in I}A_{i}$ is a topological algebra. One can then show that $\varprojlim_{i\in I}A_{i}$ is a closed subalgebra of $\prod_{i\in I}A_{i}$, so that it
follows from general topology that the profinite topology on $\varprojlim_{i\in I}A_{i}$ is a Boolean topology.
\end{remark}

We begin with a Baire-type property of profinite spaces. We deliberately do not use the group-theoretic cardinality dichotomy for profinite groups: a profinite residuated lattice need not be topologically homogeneous, and countably infinite profinite residuated lattices may occur.
\begin{proposition} \label{3.8}
Let $L$ be a profinite topological residuated lattice. If $C_{1},C_{2},\ldots$ are nonempty closed subsets of $L$ with empty interior, then
\[
L\neq\bigcup_{i=1}^{\infty}C_i.
\]
\end{proposition}

\begin{proof}
Assume that $L=\bigcup_{i=1}^{\infty}C_i$, where each $C_i$ is a nonempty closed subset of $L$ with empty interior. Then $D_i=L\setminus C_i$ is open and dense for every $i$. Let $U_0$ be a nonempty open subset of $L$. Since $D_1$ is dense, $U_0\cap D_1$ is a nonempty open set. By Theorem \ref{3.4}, profinite spaces have a base of clopen sets; hence there exists a nonempty clopen set $U_1\subseteq U_0\cap D_1$. Inductively, after $U_{n-1}$ has been chosen, the set $U_{n-1}\cap D_n$ is nonempty and open, so choose a nonempty clopen set
\[
U_n\subseteq U_{n-1}\cap D_n.
\]
Thus $U_1\supseteq U_2\supseteq\cdots$ is a decreasing sequence of nonempty closed subsets of the compact space $L$. Hence $\bigcap_{n\geq1}U_n\neq\emptyset$. But
\[
\bigcap_{n\geq1}U_n\subseteq\bigcap_{n\geq1}D_n=L\setminus\bigcup_{n\geq1}C_n=\emptyset,
\]
a contradiction.
\end{proof}

\begin{remark}\label{3.8b}
No statement of the form ``a profinite residuated lattice is either finite or uncountable'' is used in this paper. Such a statement is false in this generality. For instance, let $C_n=\{0,1,\ldots,n\}$ be the finite G\"odel chain with top element $n$ and let $p_{n+1,n}:C_{n+1}\to C_n$ be given by $p_{n+1,n}(k)=\min\{k,n\}$. These maps are homomorphisms of residuated lattices, and the inverse limit is countably infinite.
\end{remark}

\begin{lemma}\label{3.8a}
Let $\{L_i,\varphi_{ij},I\}$ be an inverse system of topological residuated lattices. Then:
\begin{enumerate}
  \item[(i)] the inverse limit $\underleftarrow{\lim}\{L_i:i\in I\}$ is a subalgebra of $\prod_{i\in I}L_i$;
  \item[(ii)] if each $L_i$ is Hausdorff, then the inverse limit is closed in $\prod_{i\in I}L_i$;
  \item[(iii)] if each $L_i$ is non-empty, compact and Hausdorff, then the inverse limit is non-empty.
\end{enumerate}
\end{lemma}

\begin{proof}
For (i), the inverse limit is defined by equations preserved by all basic operations, hence it is a subalgebra of the product. For (ii), let $a=(a_i)\in\prod_{i\in I}L_i$ be outside the inverse limit. Then there exist $r,s\in I$ with $r\geq s$ such that $\varphi_{rs}(a_r)\neq a_s$. Since $L_s$ is Hausdorff, choose disjoint open neighbourhoods $U$ and $V$ of $\varphi_{rs}(a_r)$ and $a_s$, respectively. By continuity of $\varphi_{rs}$, choose an open neighbourhood $U'$ of $a_r$ in $L_r$ such that $\varphi_{rs}(U')\subseteq U$. The basic product neighbourhood $W=\prod_{i\in I}V_i$, where $V_r=U'$, $V_s=V$ and $V_i=L_i$ for $i\neq r,s$, contains $a$ and is disjoint from the inverse limit. Hence the inverse limit is closed.

For (iii), for each $k\in I$, put
\[
A_k=\{a\in\prod_{i\in I}L_i:(\forall i\leq k)~a_i=\varphi_{ki}(a_k)\}.
\]
Each $A_k$ is closed, and the inverse limit is $\bigcap_{k\in I}A_k$. Since $I$ is directed and all $L_i$ are non-empty, every finite intersection of the sets $A_k$ is non-empty. Compactness of the product then gives $\bigcap_{k\in I}A_k\neq\emptyset$.
\end{proof}

\begin{proposition}\label{3.9}
Let
\[
L=\varprojlim_{i\in I}L_i,
\]
where $\{L_i,\varphi_{ij},I\}$ is an inverse system of finite discrete residuated lattices, and let $\pi_i:L\to L_i$ be the projection homomorphisms. Then
\[
\{\ker(\pi_i):i\in I\}
\]
is a fundamental system of open neighbourhoods of $1$ in $L$.
\end{proposition}

\begin{proof}
A basic neighbourhood of $1$ in the subspace $L\subseteq\prod_{i\in I}L_i$ contains a set of the form
\[
L\cap\bigcap_{k=1}^m\pi_{i_k}^{-1}(\{1\}),
\]
after finitely many coordinates $i_1,\ldots,i_m$ have been fixed. Since $I$ is directed, choose $j\in I$ with $i_k\leq j$ for every $k$. If $x\in\ker(\pi_j)$, then
\[
\pi_{i_k}(x)=\varphi_{j i_k}(\pi_j(x))=\varphi_{j i_k}(1)=1
\]
for every $k$. Thus $\ker(\pi_j)$ is contained in the given basic neighbourhood. Conversely, each $\ker(\pi_i)=L\cap\pi_i^{-1}(\{1\})$ is open because $L_i$ is discrete. Hence these kernels form a neighbourhood base at $1$.
\end{proof}

\begin{theorem}\label{3.10}
Let $(L,\tau)$ be a compact topological residuated lattice. Then the following conditions are equivalent:
\begin{enumerate}
  \item[(i)] $(L,\tau)$ is profinite;
  \item[(ii)] there is a set $\mathcal S$ of clopen filters on $L$ such that
  \begin{enumerate}[label=(\alph*)]
  \item $(\mathcal S,\supseteq)$ is directed;
  \item for all distinct $x,y\in L$, there is $F\in\mathcal S$ such that $x/F\neq y/F$.
  \end{enumerate}
\end{enumerate}
\end{theorem}

\begin{proof}
Assume first that (ii) holds. For each $F\in\mathcal S$, every coset $x/F$ is clopen. Indeed,
\[
x/F=\{y\in L:(x\to y)\odot(y\to x)\in F\},
\]
and this is the inverse image of the clopen set $F$ under the continuous map $y\mapsto (x\to y)\odot(y\to x)$. Since $L$ is compact and is covered by the pairwise disjoint open cosets of $F$, the quotient $L/F$ is finite. Endow each $L/F$ with the discrete topology. For $F\subseteq G$ in $\mathcal S$, the canonical map
\[
\varphi_{FG}:L/F\to L/G,\qquad \varphi_{FG}(x/F)=x/G,
\]
is a well-defined continuous homomorphism, and these maps form an inverse system.

Define
\[
\psi:L\to \varprojlim_{F\in\mathcal S}L/F,\qquad \psi(x)=(x/F)_{F\in\mathcal S}.
\]
The map $\psi$ is a continuous homomorphism. Condition (ii)(b) gives injectivity. To prove surjectivity, let $z=(z_F)_{F\in\mathcal S}$ be a compatible element of the inverse limit, where each $z_F$ is a coset of $F$ in $L$. The family $\{z_F:F\in\mathcal S\}$ has the finite intersection property: for $F_1,\ldots,F_n\in\mathcal S$, choose $H\in\mathcal S$ such that $H\subseteq F_1\cap\cdots\cap F_n$; compatibility gives $z_H\subseteq z_{F_i}$ for each $i$. Since the cosets $z_F$ are closed and $L$ is compact, $\bigcap_{F\in\mathcal S}z_F$ is nonempty. Choose $x$ in this intersection. Then $\psi(x)=z$. Thus $\psi$ is a continuous bijection from a compact space to a Hausdorff inverse limit of finite discrete spaces, hence a homeomorphism. Therefore $(L,\tau)$ is profinite.

Conversely, assume that $(L,\tau)$ is profinite, say $L\cong\varprojlim_{i\in I}L_i$ with each $L_i$ finite and discrete. Let $\pi_i:L\to L_i$ be the limit projections and put $\mathcal S=\{\ker(\pi_i):i\in I\}$. Each $\ker(\pi_i)$ is a clopen filter. If $i,j\in I$, choose $k\in I$ with $i,j\leq k$; then $\ker(\pi_k)\subseteq\ker(\pi_i)\cap\ker(\pi_j)$, so $(\mathcal S,\supseteq)$ is directed. Finally, if $x\neq y$, then some projection separates them, say $\pi_i(x)\neq\pi_i(y)$, and hence $x/\ker(\pi_i)\neq y/\ker(\pi_i)$.
\end{proof}

Since the lattice of congruences is isomorphic with the lattice of (deductive) filters in any residuated lattice, subdirect products can be characterized by (deductive) filters as follows:
\begin{proposition} \label{4.6}
A residuated lattice $L$ is a subdirect product of a family $L_i$ of residuated lattices if and only if there is a family $\{F_i\}_{i\in I}$ of filters of $L$ such that
\begin{itemize}
\item[\rm (i)] $L_i\cong L/F_i$ for each $i\in I$;
\item[\rm (ii)] $\bigcap_{i\in I}F_i=\{1\}$.
\end{itemize}
\end{proposition}
\begin{proof}
If $h:L\to\prod_{i\in I}L_i$ is a subdirect embedding, put $F_i=\ker(\pi_i\circ h)$. Then $L/F_i\cong L_i$ because $\pi_i\circ h$ is onto, and $\bigcap_iF_i=\{1\}$ because $h$ is injective and congruences are determined by the class of $1$. Conversely, given such filters, the diagonal homomorphism
\[
L\longrightarrow \prod_{i\in I}L/F_i,\qquad x\longmapsto (x/F_i)_{i\in I},
\]
is injective exactly because $\bigcap_iF_i=\{1\}$, and each coordinate projection is onto. Hence it is a subdirect embedding.
\end{proof}

Recall that a topological algebra $(L,\tau)$ is \emph{residually finite} if for any two distinct elements $x,y\in L$, there exists a finite discrete topological algebra $A$ as well as
continuous homomorphism $\varphi: L\rightarrow A$ such that $\varphi(x)\neq\varphi(y)$.
\begin{theorem}\label{3.11}
Let $(L,\tau)$ be a compact topological residuated lattice. Then the following conditions are equivalent:
\begin{enumerate}
  \item[(i)] $(L,\tau)$ is residually finite;
  \item[(ii)] $(L,\tau)$ is profinite;
  \item[(iii)] $(L,\tau)$ is topologically isomorphic to a closed subdirect product of finite discrete residuated lattices.
\end{enumerate}
\end{theorem}

\begin{proof}
$(i)\Rightarrow(iii)$. For each ordered pair $(x,y)$ with $x\neq y$, choose a continuous homomorphism $\varphi_{x,y}:L\to A_{x,y}$ into a finite discrete residuated lattice such that $\varphi_{x,y}(x)\neq\varphi_{x,y}(y)$. Replacing $A_{x,y}$ by the finite image $\varphi_{x,y}(L)$, we may assume $\varphi_{x,y}$ is onto. The diagonal map
\[
\Phi:L\to\prod_{x\neq y}\varphi_{x,y}(L),\qquad
\Phi(z)=(\varphi_{x,y}(z))_{x\neq y},
\]
is a continuous injective homomorphism and all coordinate projections onto the image factors are onto. Since $L$ is compact and the product is Hausdorff, $\Phi$ is a homeomorphism onto its compact, hence closed, image. Thus $L$ is a closed subdirect product of finite discrete residuated lattices.

$(iii)\Rightarrow(i)$. If $L$ is a subdirect product of finite discrete residuated lattices and $x\neq y$, some coordinate projection separates $x$ and $y$. This coordinate projection is a continuous homomorphism into a finite discrete residuated lattice.

$(ii)\Rightarrow(iii)$. If $L\cong\varprojlim_{i\in I}L_i$ with each $L_i$ finite and discrete, then the inverse limit is a closed subalgebra of $\prod_{i\in I}L_i$ by Lemma \ref{3.8a}. If necessary, replace each $L_i$ by the finite image $\pi_i(L)$; then all projections are onto, so the representation is subdirect.

$(iii)\Rightarrow(ii)$. Suppose $L$ is a closed subdirect product of finite discrete residuated lattices $\{L_i:i\in I\}$. Let $\mathcal S$ be the directed set of nonempty finite subsets of $I$, ordered by inclusion. For $J\in\mathcal S$, let $L_J$ be the image of $L$ under the projection to $\prod_{j\in J}L_j$. For $J\subseteq K$, let $p_{KJ}:L_K\to L_J$ be the coordinate projection. Then $\{L_J,p_{KJ},\mathcal S\}$ is an inverse system of finite discrete residuated lattices. The natural map
\[
L\to\varprojlim_{J\in\mathcal S}L_J,
\qquad x\mapsto (x|_J)_{J\in\mathcal S},
\]
is injective and continuous. To see surjectivity, take $(y_J)_{J\in\mathcal S}$ in the inverse limit and define
\[
V_J=\{x\in L:x|_J=y_J\}.
\]
Each $V_J$ is nonempty and closed in $L$, and the family $\{V_J:J\in\mathcal S\}$ has the finite intersection property because $V_{J_1}\cap\cdots\cap V_{J_n}=V_{J_1\cup\cdots\cup J_n}$. Compactness of $L$ gives $x\in\bigcap_JV_J$, and then $x$ maps to $(y_J)_J$. Thus $L$ is profinite.
\end{proof}


\section{Finiteness conditions in residuated lattices}

Finite-dimensional vector spaces do not contain infinite strictly ascending or strictly descending chains of subspaces. Analogously, many algebraic finiteness conditions can be formulated as chain conditions on distinguished substructures. In this section we discuss the finiteness conditions on filters that will be used later.\\

So far we have considered a quite arbitrary residuated lattice. To go further, however, and obtain deeper theorems we need to impose some conditions. The most convenient formulation is in terms of ``chain conditions''.
\begin{definition}\label{4.2} \emph{Let $L$ be a residuated lattice. We say $L$ satisfies the \emph{descending chain condition }($DCC$ for short) if every strictly descending chain of filters}
$$\cdots \subset F_{2}\subset F_{1} $$
\emph{is finite.}

\emph{An equivalent formulation is as follows: If}
$$\cdots \subset F_{2}\subset F_{1} $$
\emph{is an infinite descending sequence of filters, then there exists an integer $n$ such that}
$$F_{i}=F_{n}, \mbox{for~all}~ i\geq n+1.$$
\end{definition}

\begin{example}\label{4.2a}
\begin{itemize}
\item[\rm (a)] Every finite residuated lattice satisfies the descending chain condition.
\item[\rm (b)] Every simple residuated lattice satisfies the descending chain condition.

\item[\rm (c)] There are infinite residuated lattices which do not satisfy the descending chain condition. Let $L$ be a chain with an increasing sequence $a_1<a_2<\cdots<1$. Define $x\odot y=x\wedge y$ and
\[
x\rightarrow y=
\begin{cases}
1, & x\leq y,\\
y, & x>y.
\end{cases}
\]
Then each $(a_n,1]$ is a filter, and
\[
(a_1,1]\supset (a_2,1]\supset\cdots\supset (a_n,1]\supset\cdots
\]
is an infinite strictly descending chain of filters. Hence $L$ does not satisfy $DCC$.
\end{itemize}
\end{example}

\begin{definition}\label{4.3}
\emph{A residuated lattice is said to satisfy the \emph{minimal condition} if every non-empty collection $\mathcal{C}$ of filters has a minimal element, that is, if there exists a filter in $\mathcal{C}$ which is not contained in any other filter in the collection $\mathcal{C}$.}
\end{definition}

\begin{proposition}\label{4.4}
Let $L$ be a residuated lattice. Then the following assertions are equivalent:
\begin{itemize}
\item[\rm (1)] $L$ satisfies the descending chain condition;
\item[\rm (2)] $L$ satisfies the minimal condition;
\item[\rm (3)] each system of filters of $L$ has a least element.
\end{itemize}
\end{proposition}

\begin{proof}
The equivalence of (1) and (2) is the usual chain-condition argument. Indeed, if (1) fails, an infinite strictly descending chain of filters gives a nonempty collection with no minimal element. Conversely, if a nonempty collection of filters has no minimal element, choose $F_1$ in it; after $F_n$ has been chosen, choose $F_{n+1}$ strictly contained in $F_n$. This gives an infinite strictly descending chain.

Assume (2), and let $(\mathcal F,\subseteq)$ be a system of filters. Choose a minimal member $F_0$ of $\mathcal F$. For any $G\in\mathcal F$, down-directedness gives $H\in\mathcal F$ with $H\subseteq F_0\cap G$. By minimality of $F_0$, we have $H=F_0$, and hence $F_0\subseteq G$. Thus $F_0$ is the least element of $\mathcal F$.

Finally, assume (3). Any descending chain of filters is a system of filters. If it had no final constant term, it would have no least member among the displayed chain, contradicting (3). Hence (1) holds.
\end{proof}




\begin{theorem}\label{4.7}
Let $L$ be a residuated lattice and $\mathcal{F}=\{F_i:i\in I\}$ be a system of filters of $L$. Then the following statements are equivalent:
\begin{itemize}
\item[\rm (i)] $(L,\mathcal{T}_{\mathcal F})$ is a Hausdorff space;
\item[\rm (ii)] $(L,\mathcal{T}_{\mathcal F})$ is a $T_1$-space;
\item[\rm (iii)] $(L,\mathcal{T}_{\mathcal F})$ is a $T_0$-space;
\item[\rm (iv)] $\bigcap_{i\in I}F_i=\{1\}$;
\item[\rm (v)] $L$ is a subdirect product of the family $L/F_i$ of residuated lattices.
\end{itemize}
\end{theorem}

\begin{proof}
The implications (i)$\Rightarrow$(ii)$\Rightarrow$(iii) are general topological facts. We prove (iii)$\Rightarrow$(iv). If $a\neq1$ belongs to $\bigcap_iF_i$, then $(a,1)\in\theta_{F_i}$ for every $i$, because $(a\to1)\odot(1\to a)=1\odot a=a$. Hence $a/F_i=1/F_i$ for every $i$. Therefore $a$ and $1$ have the same basic neighbourhoods in $\mathcal T_{\mathcal F}$, contradicting the $T_0$ property. Thus $\bigcap_iF_i=\{1\}$.

Assume (iv), and let $x\neq y$. Then $(x\to y)\odot(y\to x)\neq1$; otherwise $x\to y=1$ and $y\to x=1$, whence $x=y$. Hence there is $i\in I$ such that $(x\to y)\odot(y\to x)\notin F_i$. Thus $x/F_i$ and $y/F_i$ are disjoint open neighbourhoods of $x$ and $y$, respectively. This proves Hausdorffness, so (iv)$\Rightarrow$(i). Finally, (iv) and (v) are equivalent by Proposition \ref{4.6}.
\end{proof}

A filter $F$ of a residuated lattice $L$ is said to be \emph{prime} if $x\vee y\in F$ implies $x\in F$ or $y\in F$ for all $x,y\in L$, see \cite{Kondo}. We denote the set of all prime filters of a residuated lattice $L$ by $Spec(L)$.

\begin{theorem}\label{4.8} \cite{Hohle}
Let $L$ be a residuated lattice. For every element $a\in L$ with $a\neq 1$ there is a prime filter $P$ in $L$ with $a\notin P$. Thus, every proper filter of $L$ is an intersection of prime filters. In particular,
$\bigcap Spec(L)=\{1\}$.
\end{theorem}



Recall that a topological space $(X,\mathcal{T})$ is an \emph{anti-discrete space} if its topology is anti-discrete, namely the only open sets in $(X,\mathcal{T})$ are $\emptyset$ and $X$.
\begin{proposition}\label{4.9a}
Let $\mathcal{F}=\{F_i:i\in I\}$ be a system of filters of a residuated lattice $L$. Then:
\begin{itemize}
\item[\rm (i)] $(L,\mathcal T_{\mathcal F})$ is discrete if and only if $\{1\}\in\mathcal F$;
\item[\rm (ii)] $(L,\mathcal T_{\mathcal F})$ is anti-discrete if and only if $\mathcal F=\{L\}$.
\end{itemize}
\end{proposition}
\begin{proof}
If $\{1\}\in\mathcal F$, then $x/\{1\}=\{x\}$ for every $x\in L$, so the topology is discrete. Conversely, if the topology is discrete, then $\{1\}$ is an open neighbourhood of $1$; hence some $F\in\mathcal F$ satisfies $1/F=F\subseteq\{1\}$, so $F=\{1\}$.

For (ii), if $\mathcal F=\{L\}$, then every basic open set is $L$, so the topology is anti-discrete. Conversely, if the topology is anti-discrete and $F\in\mathcal F$, then $F=1/F$ is a nonempty open set, hence $F=L$.
\end{proof}



We call $\mathcal{F}(L)\backslash\{\{1\}\}$ a \emph{global system of filters} if it is a system of filters, equivalently, if finite intersections of non-trivial filters are again non-trivial. Recall that a congruence $\theta$ on an algebra is a \emph{factor congruence} if there is a congruence $\theta^{*}$ such that $\theta\cap\theta^{*}=\triangle$, $\theta\vee\theta^{*}=\nabla$ and $\theta$ permutes with $\theta^{*}$; see \cite{Burris}. An algebra is \emph{directly indecomposable} if it is not isomorphic to a direct product of two nontrivial algebras.
\begin{proposition}\label{4.11}
Let $L$ be a residuated lattice. The following conditions are equivalent:
\begin{enumerate}
  \item[(1)] $\mathcal{F}(L)\backslash\{\{1\}\}$ is a global system of filters;
  \item[(2)] the intersection of any two non-trivial filters of $L$ is not equal to $\{1\}$;
  \item[(3)] there are no two non-trivial congruences $\alpha,\beta\in Con(L)$ such that $\alpha\cap\beta=\triangle$.
\end{enumerate}
Moreover, each of these conditions implies that $L$ is directly indecomposable.
\end{proposition}

\begin{proof}
The equivalence of (1) and (2) follows because finite intersections are obtained by iterating binary intersections. The equivalence of (2) and (3) follows from the filter-congruence correspondence in Theorem \ref{4.1}: the intersection of congruences corresponds to the intersection of their associated filters. Finally, if $L$ were a direct product of two nontrivial algebras, the two projection kernels would be non-trivial congruences with intersection $\triangle$, contradicting (3). Equivalently, this is the standard factor-congruence criterion in \cite[Corollary 7.7]{Burris}. The converse implication from direct indecomposability to (1) is not asserted.
\end{proof}

An algebra $\mathbf{A}$ is \emph{subdirectly irreducible} iff $\mathbf{A}$ is trivial or there is a smallest nontrivial congruence, see \cite{Burris}. By Theorem \ref{4.1}, a nontrivial residuated lattice $L$ is subdirectly irreducible iff there is a least element in the poset $\mathcal{F}(L)\backslash\{\{1\}\}$.

\begin{lemma} \label{4.12}
If a nontrivial residuated lattice $L$ is subdirectly irreducible, then $\mathcal{F}(L)\backslash\{\{1\}\}$ is a global system of filters and the topology induced by this global system is non-discrete.
\end{lemma}

\begin{proof}
Let $F_0$ be the least non-trivial filter of $L$. Every non-trivial filter contains $F_0$, so every finite intersection of non-trivial filters also contains $F_0$ and is therefore non-trivial. Hence the non-trivial filters form a global system. Since $\{1\}$ is not a member of this system, Proposition \ref{4.9a} shows that the induced topology is not discrete.
\end{proof}

\begin{theorem}\label{4.13}
Let $L$ be a nontrivial residuated lattice satisfying $DCC$. Then $L$ is subdirectly irreducible if and only if $\mathcal{F}(L)\backslash\{\{1\}\}$ is a global system of filters. In this case, the topology induced by the global system is a non-discrete zero-dimensional linear topology.
\end{theorem}

\begin{proof}
The forward implication follows from Lemma \ref{4.12}. Conversely, assume that $\mathcal{F}(L)\backslash\{\{1\}\}$ is a global system. By Proposition \ref{4.4}, this system has a least element $F_0$. Thus $F_0$ is the least non-trivial filter of $L$, and hence $L$ is subdirectly irreducible. The last assertion follows again from Proposition \ref{4.9a} and the construction of filter-induced topologies.
\end{proof}



Recall that a join $\bigvee_{i=1}^{m}x_i$ in a lattice is \emph{irredundant} if it is such that, for every $k\in\{1,\ldots,m\}$, $\bigvee_{i=1}^{m}x_i>x_i\vee\cdots\vee x_{k-1}\vee x_{k+1}\vee\cdots \vee x_m$. Roughly speaking,
a join is irredundant if the removal of any term results in something smaller.

\begin{proposition}\label{5.6} If $L$ is a residuated lattice that satisfies the descending chain condition then every element of $\mathcal{F}(L)\backslash\{\{1\}\}$ can be expressed uniquely as an irredundant join of join-irreducible filters.
\end{proposition}

\begin{proof}
The lattice $(\mathcal{F}(L),\subseteq)$ is distributive. Under $DCC$, the decomposition theorem for distributive lattices with descending chain condition applies; see \cite[Theorem 5.2]{Blyth}. Hence each non-bottom element has a unique irredundant decomposition as a finite join of join-irreducible elements.
\end{proof}

In the power set lattice $\mathcal {P}(X)$, the join-irreducible elements are exactly the singleton subsets. Let $\mathcal{JF}(L)$ denote the set of join-irreducible filters of a residuated lattice $L$.
Generally, the set $\mathcal{JF}(L)$ of all join-irreducible filters of $L$ is not a down-directed set under the set inclusion. In the following we show that $(\mathcal{JF}(L)\cup \{\{1\}\},\supseteq)$ is a directed set under some finiteness conditions.
\begin{lemma}\label{5.10}
If $L$ is a residuated lattice that satisfies the descending chain condition, then $(\mathcal{JF}(L)\cup\{\{1\}\},\supseteq)$ is cofinal in $(\mathcal{F}(L),\supseteq)$.
\end{lemma}

\begin{proof}
Suppose that $F,G\in \mathcal{JF}(L)$. If $F\cap G=\{1\}$, then there is nothing to do, otherwise, without loss of generality we can assume that $F\cap G$ is not a join-irreducible filter, thus, by Proposition \ref{5.6}, we
have $F\cap G=\bigvee_{i=1}^{n}F_{i}$,
where $F_{i}\in\mathcal{JF}(L)$. It follows that $F_{i}\subset F\cap G$, whence $(\mathcal{JF}(L),\supseteq)$ is a directed poset. Let $F$ be a filter of $L$. Then, by Proposition \ref{5.6}, there are join-irreducible filters
$J_{1},\ldots,J_{n}$ such that $F=\bigvee_{i=1}^{n}J_{i}$. Thus there some join-irreducible filter $J_{j}$ such that $J_{j}\subseteq F$. Therefore $(\mathcal{JF}(L),\supseteq)$ is cofinal in $(\mathcal{F}(L),\supseteq)$.
\end{proof}


\begin{theorem}\label{5.11}
Let $L$ be a finite residuated lattice. Then $$\widehat{L}=\varprojlim_{F\in \mathcal{JF}(L)\cup\{\{1\}\}}L/F.$$
\end{theorem}

\begin{proof}
Since every finite residuated lattice satisfies the descending chain condition, the theorem follows from Proposition \ref{3.3} and Lemma \ref{5.10}.
\end{proof}

\section{The existence of Hausdorff topological residuated lattices}
In this section, we investigate linear topological residuated lattices and discuss the existence of Hausdorff topological residuated lattices.\\

Now, we consider a special system of filters of a residuated lattice $L$, namely the set of all filters of finite index, denoted by $\mathcal{F}(L)_\mathrm{fin}$. This is a down-directed set: if $F,G\in\mathcal F(L)_\mathrm{fin}$, then $F\cap G$ is a filter and $L/(F\cap G)$ embeds into $L/F\times L/G$, so $F\cap G$ again has finite index. Thus, $(L,\mathcal{T}_{\mathcal{F}(L)_\mathrm{fin}})$ is a zero-dimensional linear topological residuated lattice. Like the linear topology of abelian groups \cite{Fuchs}, this linear topology in a residuated
lattice $L$ is called a \emph{finite index topology} of $L$. Recall that an algebra $A$ is \emph{finitely approximable} if $A$ is isomorphic to a subalgebra of a product of finite algebras \cite{Malcev}, the class of
finitely approximable residuated lattices is denoted by $\mathbf{Ap}\mathcal{RL}$. It follows that
$A$ is finitely approximable iff $A$ is a subdirect product of its finite homomorphic images.
Thus we have that a residuated lattice $L$ is finitely approximable if and only if $\bigcap\mathcal{F}(L)_\mathrm{fin}=\{1\}$.\\

 As a consequence, we have the following characterization of finitely approximable residuated lattices by
using zero-dimensional linear topological residuated lattices.

\begin{theorem}\label{8.2}
A residuated lattice is finitely approximable if and only if its finite index topology is Hausdorff.
\end{theorem}

\begin{proof}
It follows from the above statement and Theorem \ref{4.7}.
\end{proof}

Let $(L,\mathcal{U})$ be a linear topological residuated lattice. Then there exists a base $\beta$ for the
topology $\mathcal{U}$ on $L$ such that any element containing 1 of $\beta$ is a filter of $L$. Let $X$ be the set of all elements containing
1 of $\beta$, $Y$ be the set of all finite intersections of elements of $X$. Then $Y$ is a system of filters of $L$. The topology induced by the system $Y$ of filters is denoted by $\mathcal{V}$.\\

In what follows, we give a correction to \cite[Theorem 3.7]{Yang}, which answers an open problem raised in \cite{Zahiri}. Roughly speaking, the open problem is whether every linear topological residuated lattice (or BL-algebra) can be induced by a system of filters. For more details, we refer the reader to \cite[Remark 3.5]{Zahiri} or \cite[Introduction]{Yang}.
The main issue with the proof of \cite[Theorem 3.7]{Yang} is that, given a system of filters in a residuated lattice--even one closed under finite intersections--such a system does not necessarily admit a minimal element (let alone a least element). We construct an explicit counterexample below.

\begin{example}\label{7.14}
For each $a\in[0,1)$, the set $(a,1]$ is a filter of the G\"odel algebra $\mathcal I$.
Moreover, for any finite family $\{(a_i,1]\}_{i=1}^n\subseteq\mathcal F$, one has
\[
\bigcap_{i=1}^n (a_i,1]=(\max\{a_1,\dots,a_n\},1]\in\mathcal F.
\]
Hence $\mathcal F$ is closed under finite intersections. However, $\mathcal F$ is not closed under arbitrary intersections. Indeed,
\[
\bigcap_{n=1}^\infty \left(1-\frac1n,1\right]=\{1\}\notin\mathcal F.
\]
\end{example}

\begin{remark}\label{7.14a}
Finite intersection closure only guarantees that the intersection of finitely many filters lies in the set, but it does not constrain whether the limit of an infinite descending chain (infinite intersection) belongs to the set. This gap allows the construction of filter sets with no minimal elements.
The role of the descending chain condition is precisely to remove this gap.  By Proposition \ref{4.4}, if a residuated lattice $L$ satisfies $DCC$, then every system of filters of $L$ has a least element.  Hence the infinite-intersection step used in the proof of \cite[Theorem 3.7]{Yang} can be replaced by a single least filter.  We record the corrected form below and include the details, in order not to rely on the unrestricted version of that theorem.
\end{remark}

\begin{theorem}\label{5.6a}
Let $L$ be a residuated lattice satisfying the descending chain condition. If $(L,\mathcal U)$ is a zero-dimensional linear topological residuated lattice, then there exists a filter $F_0$ of $L$ such that
\[
\mathcal U=\mathcal T_{F_0}.
\]
In particular, on a residuated lattice satisfying $DCC$, every zero-dimensional linear topology is induced by a single filter, and hence by a singleton system of filters.
\end{theorem}

\begin{proof}
Choose a base $\beta$ for the linear topology $\mathcal U$ such that every member of $\beta$ containing $1$ is a filter. Put
\[
X_{\mathcal U}=\{F\in\beta:1\in F\},
\]
and let $Y_{\mathcal U}$ be the family of all finite intersections of members of $X_{\mathcal U}$. Then $Y_{\mathcal U}$ is a system of filters of $L$. Let
\[
\mathcal V=\mathcal T_{Y_{\mathcal U}}.
\]

We first prove that $\mathcal V\subseteq\mathcal U$. Indeed, every member of $Y_{\mathcal U}$ is a $\mathcal U$-open filter. If $F\in Y_{\mathcal U}$ and $x\in L$, then Lemma \ref{7.19} gives that the coset $x/F$ is $\mathcal U$-open. Since the sets $x/F$, with $F\in Y_{\mathcal U}$ and $x\in L$, form a base for $\mathcal V$, it follows that
\[
\mathcal V\subseteq\mathcal U.
\]

We now prove the reverse inclusion. Since $L$ satisfies $DCC$, Proposition \ref{4.4} implies that the filter system $Y_{\mathcal U}$ has a least element; denote it by $F_0$. Thus
\[
F_0\subseteq F\qquad(F\in Y_{\mathcal U}).
\]
Consequently, for every $x\in L$,
\[
x/F_0\subseteq x/F\qquad(F\in Y_{\mathcal U})
\]
and therefore
\[
\bigcap_{F\in Y_{\mathcal U}}x/F=x/F_0.        \tag{1}
\]

Let $O$ be a clopen member of a zero-dimensional base of $\mathcal U$ and let $x\in O$. We use only the local part of the proof of \cite[Theorem 3.7]{Yang}, namely the closure calculation
\[
\bigcap_{F\in Y_{\mathcal U}}x/F\subseteq O.        \tag{2}
\]
This local calculation is independent of the global finite-reduction step in that proof. The latter step is not valid in general: one cannot replace an arbitrary intersection of filters by a member of the same filter system without an additional finiteness assumption. In the present theorem this missing step is supplied exactly by $DCC$, because Proposition \ref{4.4} gives a least member $F_0$ of $Y_{\mathcal U}$, and hence (1) turns the arbitrary intersection of cosets into the single coset $x/F_0$.

Combining (1) and (2), we obtain
\[
x/F_0\subseteq O.
\]
Thus every clopen basic neighbourhood $O$ of every point $x$ contains the $\mathcal V$-basic neighbourhood $x/F_0$ of $x$. Since $\mathcal U$ has a clopen base, this proves
\[
\mathcal U\subseteq\mathcal V.
\]
Together with the first inclusion, we get $\mathcal U=\mathcal V$.

Finally, since $F_0$ is the least element of $Y_{\mathcal U}$, the topology induced by the whole system $Y_{\mathcal U}$ coincides with the simple topology induced by $F_0$. Explicitly, if $U\in\mathcal T_{Y_{\mathcal U}}$ and $x\in U$, then for some $F\in Y_{\mathcal U}$ we have $x/F\subseteq U$; as $F_0\subseteq F$, we have $x/F_0\subseteq x/F\subseteq U$. Hence $U\in\mathcal T_{F_0}$. The reverse inclusion is immediate from $F_0\in Y_{\mathcal U}$. Therefore
\[
\mathcal U=\mathcal V=\mathcal T_{Y_{\mathcal U}}=\mathcal T_{F_0}.
\]
This proves the theorem.
\end{proof}

\begin{corollary}\label{5.6b}
Let $L$ be a residuated lattice satisfying the descending chain condition ($DCC$). Then a topology $\mathcal U$ on $L$ is a zero-dimensional linear topology if and only if there exists a system of filters $\mathcal F$ on $L$ such that
\[
\mathcal U=\mathcal T_{\mathcal F}.
\]
Moreover, $\mathcal F$ may be replaced by a single filter; hence $\mathcal U=\mathcal T_F$ for some filter $F$ of $L$.
\end{corollary}

\begin{proof}
The topology induced by a system of filters is zero-dimensional and linear by the construction in Section 2. Conversely, if $\mathcal U$ is a zero-dimensional linear topology, Theorem \ref{5.6a} gives a filter $F$ of $L$ such that $\mathcal U=\mathcal T_F$. Taking $\mathcal F=\{F\}$ proves the assertion.
\end{proof}

 Let $F$ be a filter of a residuated lattice $L$. We can get a zero-dimensional linear topological residuated lattice $(L,\mathcal{T}_{F})$, where $\mathcal{T}_{F}=\{U\subseteq L: \forall x\in U, x/F\subseteq U\}$, see \cite{Yang}.
For simplicity, such a zero-dimensional linear topological residuated lattice is called a \emph{simple linear topological residuated lattice} when is no danger of confusion. Indeed, such a topology is also called the \emph{$I$-adic topology} on MV-algebras in \cite{Hoo}.\\

\begin{theorem}\label{7.12}
Let $L$ be a residuated lattice satisfying the descending chain condition ($DCC$). Then each zero-dimensional linear topology $\mathcal U$ of $L$ is simple. More precisely, there exists a filter $F_0$ of $L$ such that
\[
\mathcal U=\mathcal T_{F_0}.
\]
Equivalently, if $Y_{\mathcal U}$ is the system of finite intersections of open filter-neighbourhoods of $1$, then $Y_{\mathcal U}$ has a least member $F_0$ and
\[
\mathcal U=\mathcal T_{Y_{\mathcal U}}=\mathcal T_{F_0}=\sup\{\mathcal T_F:F\in Y_{\mathcal U}\}.
\]
\end{theorem}

\begin{proof}
This is exactly Theorem \ref{5.6a}. The last equality follows because $F_0\in Y_{\mathcal U}$ and $F_0\subseteq F$ for all $F\in Y_{\mathcal U}$; hence every topology $\mathcal T_F$ with $F\in Y_{\mathcal U}$ is contained in $\mathcal T_{F_0}$, while $\mathcal T_{F_0}$ itself occurs in the displayed family.
\end{proof}

\begin{definition}\label{7.17}
\emph{Let $L$ be a residuated lattice and let $(\mathcal F_1,\subseteq)$ and $(\mathcal F_2,\subseteq)$ be two systems of filters of $L$. We say that $\mathcal F_1$ and $\mathcal F_2$ are \emph{equivalent} if they are coinitial with respect to inclusion: for every $F_1\in\mathcal F_1$ there exists $F_2\in\mathcal F_2$ such that $F_2\subseteq F_1$, and for every $F_2\in\mathcal F_2$ there exists $F_1\in\mathcal F_1$ such that $F_1\subseteq F_2$.}
\end{definition}

\begin{proposition}\label{7.18}
In a residuated lattice $L$, equivalent systems of filters have the same topology.
\end{proposition}

\begin{proof}
For any $x\in U\in\mathcal{T}_{\mathcal{F}_1}$, then there is $F_1\in\mathcal{F}_1$ such that $x/F_1\subset U$. For $F_1\in\mathcal{F}_1$, by Definition \ref{7.17}, there is $F_2\in\mathcal{F}_2$ such that
$F_2\subseteq F_1$. Thus, we have $x\in x/F_2\subseteq x/F_1\subseteq U$. Hence $\mathcal{T}_{\mathcal{F}_1}\subseteq\mathcal{T}_{\mathcal{F}_2}$. From the symmetrical characteristic of Definition \ref{7.17}, it can be
concluded that $\mathcal{T}_{\mathcal{F}_1}=\mathcal{T}_{\mathcal{F}_2}$.
\end{proof}

\begin{example} \label{4.5} Let $\odot$ and $\rightarrow$ on the real unit interval $I=[0,1]$ be defined as follows:
\begin{center} $x\odot y=\min\{x,y\}$ and
$x\rightarrow y=
\begin{cases}
1, &  x\leq y,\\
y, & otherwise.\\
\end{cases}$
\end{center}
Then $\mathcal{I}=(I,\min,\max,\odot,\rightarrow,0,1)$ is a $BL$-algebra (called the G\"{o}del algebra). For every $a\in[0,1)$, the interval $(a,1]$ is an upward closed subset containing $1$, and it is closed under $\odot=\min$; hence it is a filter. Moreover,
\[
(a,1]\cap(b,1]=(\max\{a,b\},1]
\]
for all $a,b\in[0,1)$, so $\mathcal{F}=\{(a,1]\mid a\in[0,1)\}$ is a system of filters. Since
\[
(0,1]\supsetneq (1/2,1]\supsetneq (2/3,1]\supsetneq\cdots,
\]
$\mathcal I$ does not satisfy $DCC$.
\end{example}

Using the above example, we can identify the precise gap in \cite[Theorem 3.10]{Yang}: the proof replaces a given filter system by a larger family obtained using arbitrary intersections, but the two systems need not be equivalent. In fact, in the proof of \cite[Theorem 3.10]{Yang}, suppose that $(L,\mathcal{T})$ is a linear topological residuated lattice such that $(L,\mathcal{T})$ is a zero-dimensional space. Then there exists a base $\beta$ for the topology $\mathcal{T}$ on $L$ such that any element containing 1 of $\beta$ is a filter of $L$. Let $X$ be the set of all elements containing 1 of $\beta$ and $Y$ be the set of all intersections of elements of $X$. Then $(X,\subseteq)$ and $(Y,\subseteq)$ are two systems of filters of a residuated lattice $L$. In general, they are not equivalent. Then the topologies induced by them are not same. For example,
taking $X=\mathcal{F}$ in Example \ref{4.5} and using the construction of $Y$, we have $\{1\}\in Y$. Thus, $\mathcal{T}_{Y}$ is a discrete topology. But $\mathcal{T}_{X}$ is not a discrete topology. Thus, $\mathcal{T}_{X}\neq\mathcal{T}_{Y}$, that is, $\mathcal{T}\neq\mathcal{V}$ in the proof of \cite[Theorem 3.10]{Yang}.

This example also shows that, without a finiteness assumption such as $DCC$, the canonical map from filters to zero-dimensional linear topologies need not be surjective. Indeed, in the topology $\mathcal T_X$ every point $x<1$ is isolated, because one may choose $a\in(x,1)$ and then $x/(a,1]=\{x\}$. On the other hand, $1$ is not isolated, since every basic neighbourhood of $1$ contains some interval $(a,1]$ with $a<1$. We claim that $\mathcal T_X$ is not equal to $\mathcal T_G$ for any single filter $G$ of the G\"odel algebra. If $G=\{1\}$, then $\mathcal T_G$ is discrete, contrary to the fact that $1$ is not isolated in $\mathcal T_X$. If $G$ contains some element $c<1$, then the point $c$ is not isolated in $\mathcal T_G$, since $c/G=G$ contains elements larger than $c$. This contradicts the fact that every point below $1$ is isolated in $\mathcal T_X$. Hence $\mathcal T_X$ is a zero-dimensional linear topology which is not induced by a single filter. Therefore the mapping $\varphi$ from $\mathcal{F}(L)$ to $ZLTRL(L)$ defined by $\varphi(F)=(L,\mathcal{T}_{F})$ for all $F$ in $\mathcal{F}(L)$ is not surjective in this example, where $ZLTRL(L)$ denotes the set of all zero-dimensional linear topological residuated lattices on the underlying residuated lattice $L$, see \cite{Yang}.

Next, we will correct \cite[Theorem 3.10]{Yang} in the following theorem.

\begin{theorem}\label{7.16}
Let $L$ be a residuated lattice satisfying the descending chain condition. Then the canonical map
\[
\Phi:\mathcal F(L)\longrightarrow ZLTRL(L),\qquad F\longmapsto (L,\mathcal T_F),
\]
from filters to zero-dimensional linear topological residuated lattices on the underlying set $L$ is bijective.
\end{theorem}

\begin{proof}
For every filter $F$, the topology $\mathcal T_F$ is zero-dimensional and linear, so $\Phi$ is well defined. It is injective because $F=1/F$ is the least neighbourhood of $1$ in the topology $\mathcal T_F$; hence $\mathcal T_F=\mathcal T_G$ implies $F=G$. For surjectivity, let $(L,\mathcal U)$ be a zero-dimensional linear topological residuated lattice. By Theorem \ref{5.6a}, there exists a filter $F_0$ of $L$ such that $\mathcal U=\mathcal T_{F_0}$. Thus $(L,\mathcal U)$ lies in the image of $\Phi$.
\end{proof}

\begin{remark}\label{7.16a}
Theorem \ref{7.16} is a statement about the canonical map $F\mapsto\mathcal T_F$. It is not a purely cardinality statement: equality of cardinalities alone does not imply that this particular map is surjective.
\end{remark}

\begin{theorem}\label{7.22}
Let $L$ be a non-trivial residuated lattice satisfying the descending chain condition. Then $L$ is subdirectly irreducible if and only if it has a largest non-trivial zero-dimensional linear topology.
\end{theorem}

\begin{proof}
By Theorem \ref{7.16}, zero-dimensional linear topologies on $L$ are exactly the topologies $\mathcal T_F$ with $F\in\mathcal F(L)$. Moreover, for filters $F,G$ one has
\[
\mathcal T_G\subseteq\mathcal T_F \quad\text{if and only if}\quad F\subseteq G.
\]
If $L$ is subdirectly irreducible, let $F_0$ be its least non-trivial filter. Then $F_0\subseteq G$ for every non-trivial filter $G$, and hence $\mathcal T_G\subseteq\mathcal T_{F_0}$. Thus $\mathcal T_{F_0}$ is the largest non-trivial zero-dimensional linear topology. Conversely, if $\mathcal T_F$ is largest among non-trivial zero-dimensional linear topologies, then the displayed equivalence gives $F\subseteq G$ for every non-trivial filter $G$. Hence $F$ is the least non-trivial filter, so $L$ is subdirectly irreducible.
\end{proof}

\begin{lemma}\label{7.19}
Let $(L,\mathcal{U})$ be a topological residuated lattice. If $F$ is a filter of $L$, then $F$ is open (closed) if and only if $a/F$ is open (closed) for any $a\in L$.
\end{lemma}

\begin{proof}
There is no loss of generality in assuming that $F$ is an open filter of $L$. Let $a\in L$. We define two maps $l_{a}:L\rightarrow L$ and $g_{a}:L\rightarrow L$ by $l_{a}(x)=a\rightarrow x$ and $g_{a}(x)=x\rightarrow a$
for any $x\in L$. Then clearly, $l_{a}$ and $g_{a}$ are continuous maps and so $l_{a}^{-1}(F)$, $g_{a}^{-1}(F)$ are open in $(L,\mathcal{U})$. Since $a/F=\{x\in L:x\rightarrow a\in F ~and~a\rightarrow x\in F\}=l_{a}^{-1}(F)\cap g_{a}^{-1}(F)$, then we have $a/F\in\mathcal{U}$. Conversely, it follows the fact $F=1/F$.

\end{proof}

We state next an easy consequence of compactness which is similar to topological groups \cite{Ribes}.

\begin{theorem}\label{7.20}
In a compact zero-dimensional linear topological residuated lattice $(L,\mathcal{U})$, a filter $G$ is open if and only if $G$ is closed and of finite index.
\end{theorem}

\begin{proof}
Suppose first that $G$ is open. Since $\mathcal U$ is a linear topology, there is a base $\beta$ for $\mathcal U$ such that every member of $\beta$ containing $1$ is a filter. As $G$ is an open neighbourhood of $1$, we may choose $F\in\beta$ with
\[
1\in F\subseteq G.
\]
Then $F$ is an open filter. The cosets $x/F$ cover $L$. Since $(L,\mathcal U)$ is compact and, by Lemma \ref{7.19}, each coset $x/F$ is open, there exist $x_1,\ldots,x_n\in L$ such that
\[
L=\bigcup_{i=1}^n x_i/F.
\]
Since $F\subseteq G$, we have $x_i/F\subseteq x_i/G$ for each $i$, and hence
\[
L=\bigcup_{i=1}^n x_i/G.
\]
Thus $G$ has finite index. After deleting repeated cosets, write
\[
L=G\cup x_2/G\cup\cdots\cup x_m/G.
\]
Since $G$ is open, Lemma \ref{7.19} implies that each coset $x_i/G$ is open. Therefore
\[
L\setminus G=\bigcup_{i=2}^m x_i/G
\]
is open, and so $G$ is closed.

Conversely, suppose that $G$ is closed and of finite index. Then there are finitely many distinct cosets of $G$ such that
\[
L=G\cup x_2/G\cup\cdots\cup x_m/G.
\]
By Lemma \ref{7.19}, since $G$ is closed, each coset $x_i/G$ is closed. Hence
\[
L\setminus G=\bigcup_{i=2}^m x_i/G
\]
is closed. It follows that $G$ is open. This proves the theorem.
\end{proof}

\begin{proposition}\label{7.21}
Let $\mathcal{F}(L)\setminus\{\{1\}\}$ be a global system of filters of a residuated lattice $L$. If $(L,\mathcal{T}_{\mathcal{F}(L)\setminus\{\{1\}\}})$ is compact and Hausdorff, then $L$ is finitely approximable.
\end{proposition}

\begin{proof}
It follows from Theorem \ref{4.7} and Theorem \ref{7.20}.
\end{proof}

By a chain of prime filters of a residuated lattice $L$ we mean a finite strictly increasing sequence $P_{0}\subset P_{1}\subset\cdots\subset P_{n}$, the \emph{length} of the chain is $n$.
\begin{definition}\label{7.23}
\emph{We define the \emph{dimension} of a residuated lattice $L$ to be the supremum of the lengths of all chains of prime filters in $L$: it is an integer $n\geq 0$, or $+\infty$ (assuming $L$ is nontrivial). }
\end{definition}

Note that such a dimension is known as the \emph{Krull dimension} in commutative rings (cf. \cite{Atiyah}).
\begin{proposition}\label{7.24}
If a residuated lattice $L$ is subdirectly irreducible, then no filter-induced Hausdorff linear topology on $L$ is non-discrete. If, in addition, $L$ satisfies $DCC$, then $L$ has no non-discrete Hausdorff zero-dimensional linear topology.
\end{proposition}

\begin{proof}
Let $\mathcal F$ be a system of filters such that $\mathcal T_{\mathcal F}$ is Hausdorff. By Theorem \ref{4.7}, $\bigcap\mathcal F=\{1\}$. In a subdirectly irreducible algebra, the least congruence is completely meet-irreducible; under the filter-congruence correspondence this says that if an intersection of filters is $\{1\}$, then one of the filters is $\{1\}$. Hence $\{1\}\in\mathcal F$, and Proposition \ref{4.9a} implies that $\mathcal T_{\mathcal F}$ is discrete. If $L$ also satisfies $DCC$, Corollary \ref{5.6b} reduces every zero-dimensional linear topology to the filter-induced case.
\end{proof}

\begin{theorem}\label{7.25}
Let $L$ be a residuated lattice. Suppose that there exists a system of filters $\mathcal F$ such that
\[
\{1\}\notin\mathcal F \qquad\text{and}\qquad \bigcap_{F\in\mathcal F}F=\{1\}.
\]
Then $\mathcal T_{\mathcal F}$ is a non-discrete Hausdorff linear topology on $L$.
\end{theorem}

\begin{proof}
By construction, $\mathcal T_{\mathcal F}$ is a linear topology on $L$. The equality $\bigcap_{F\in\mathcal F}F=\{1\}$ gives Hausdorffness by Theorem \ref{4.7}. Since $\{1\}\notin\mathcal F$, Proposition \ref{4.9a} shows that the topology is not discrete.
\end{proof}

\begin{example}\label{7.26}
Let $\mathcal I=([0,1],\min,\max,\odot,\rightarrow,0,1)$ be the G\"odel algebra of Example \ref{4.5}, and let
\[
\mathcal F=\{(a,1]:a\in[0,1)\}.
\]
Then $\mathcal F$ is a system of filters, $\{1\}\notin\mathcal F$, and
\[
\bigcap_{a\in[0,1)}(a,1]=\{1\}.
\]
Therefore Theorem \ref{7.25} yields a non-discrete Hausdorff zero-dimensional linear topology $\mathcal T_{\mathcal F}$ on $\mathcal I$. For $a\in[0,1)$, a direct computation gives
\[
x/(a,1]=
\begin{cases}
\{x\}, & x\leq a,\\
(a,x], & x>a.
\end{cases}
\]
Consequently, $\mathcal T_{\mathcal F}$ is the topology in which every point $x<1$ is isolated, while $1$ has the neighbourhood base
\[
(a,1]\qquad(a<1).
\]
In particular, $\mathcal T_{\mathcal F}$ is Hausdorff and non-discrete.
\end{example}


\section{Conclusions and future research topics}
In this paper, we studied linear topologies on residuated lattices induced by systems of filters. The main results give Hausdorff criteria in terms of intersections of filters, characterize compact profinite topological residuated lattices through residual finiteness and closed subdirect products of finite discrete residuated lattices, and clarify the role of the descending chain condition in the representation of zero-dimensional linear topologies. In particular, under $DCC$ the canonical correspondence between filters and zero-dimensional linear topologies is bijective.

Several problems remain open. Given a profinite algebra $A$, one may ask whether there is an algebra $B$ such that $A\cong \widehat{B}$. In lattice theory this problem is related to Gr\"{a}tzer's celebrated problem of representable posets, see \cite[Problems 34--35]{Gratzer}. Another natural question, going back to Johnstone's perspective on Stone topological algebras \cite[Sec. VI.2.6]{Johnstone}, is to determine when a Stone topological residuated lattice is profinite. We leave these questions for future work.\\

\medskip
\noindent\textbf{Acknowledgments}
\medskip\\
\indent This research is supported by a grant of National Natural Science Foundation of China (12571490).

\end{document}